\documentclass[11pt,letterpaper]{article}
\usepackage[pagewise]{lineno}
\usepackage[english]{babel}
\usepackage{amsmath}
\usepackage{amsfonts}
\usepackage{amssymb}
\usepackage{amsthm}
\usepackage{graphicx}
\usepackage{mathtools}
\usepackage{vmargin}
\setpapersize{A4}

\newtheorem{theorem}{Theorem}[section]
\newtheorem{definition}[theorem]{Definition}
\newtheorem{proposition}[theorem]{Proposition}
\newtheorem{corollary}[theorem]{Corollary}
\newtheorem{lemma}[theorem]{Lemma}

\newtheorem{remark}[theorem]{Remark}

\theoremstyle{definition}

\def\PSL{\operatorname{PSL}}

\def\covr{\operatorname{Cov^R_0}}
\def\covs{\operatorname{Cov^S_0}}
\def\covp{\operatorname{Cov^P_0}}
\def\covq{\operatorname{Cov^Q_0}}
\def\supp{\operatorname{supp}}

\def\critpt{\operatorname{CritPt}}

\def\interior{\operatorname{int}}
\def\blow{\operatorname{Bl}}

\def\orb{\operatorname{Orb}}
\def\Arg{\operatorname{Arg}}

\newcommand\restr[2]{{
  \left.\kern-\nulldelimiterspace 
  #1 
  \vphantom{\big|} 
  \right|_{#2} 
  }}

\def\F{\mathcal{F}_a}
\def\J{\operatorname{J}_a}

\def\f{f_a}

\usepackage{pgfplots}
\pgfplotsset{compat=1.15}
\usepackage{mathrsfs}
\usetikzlibrary{arrows}

\begin{document}

\title{MME for Mating Correspondences}
\date{\today}
\author{V. Matus de la Parra}


\maketitle

\begin{abstract}
We use the measure theoretical entropy from \cite{VivSir22} and the topological entropies from \cite{DinSib08II} and \cite{KelTen17} to prove that the equidistribution measures $\mu_-$ and $\mu_+$ found in \cite{Mat23} maximize entropy for the $1$-parameter family $\lbrace\F\rbrace_a$ studied in \cite{BulLom20I}. We find as well a measure of maximal entropy for the composition of two deleted covering correspondences studied in \cite{Bul00}.
\end{abstract}

\section*{Introduction}

Mating correspondences arise as a way to fit both sides of Sullivan's dictionary into one dynamical and algebraic object. In \cite{BulPen94}, Bullett and Penrose showed that for $a\in(1,4]$ the correspondence $\F=\J\circ\covq$ is a ``mating" between the modular group $\PSL_2(\mathbb{Z})$ and a quadratic rational map $z^2+c$, where $\J(z)=\frac{(a+1)z-2a}{2z-(a+1)}$ and $\covq(z)$ is the deleted covering correspondence of $Q(z)=z^3-3z$.

On the other hand, in \cite{BulHar00}, Bullett and Harvey proved that given a quadratic polynomial $z^2+c$ with connected Julia set, and a faithful discrete representation $r$ of $C_2*C_3$ in $\PSL_2(\mathbb{Z})$ having connected regular set, there exists a (2:2) holomorphic correspondence $F$ that is a ``mating" between $z^2+c$ and $r$. They find an implicit formula for such correspondence, and the formula is similar to the one of \cite{BulPen94}. In the case where $z^2+c$ is hyperbolic with Julia set is homeomorphic to a closed disk, it was proven by Bharali and Sridharan in \cite{BhaSri16} that this correspondence has the equidistribution property in the sense that there exists an invariant probability measure $\mu_F$ and an open set $U$ containing the limit set so that
$$\frac{1}{2^n}(F^n)^*\delta_z\to\mu_F$$
weakly, as $n\to\infty$, for all $z\in U$. Here $(F^n)^*$ represents the pullback operator induced by $F^n$.
Later, Bullett and Lomonaco proved in \cite{BulLom20I} that for $a\in\mathcal{K}$ ($\supset(1,4]$), $\F$ is a ``mating" between a parabolic quadratic rational map $P_A(z)$ and $\PSL_2(\mathbb{Z})$. 



In \cite{DinKauWu20}, Dinh, Kaufmann and Wu introduced the notion of weakly-modular correspondences and characterized them using pullback operators. Modular correspondences are weakly-modular, and equidistribution has been shown for non weakly-modular correspondences in \cite{DinKauWu20} and for modular correspondences in \cite{CloOhUll01} by Clozel, Oh and Ullmo. In \cite{Mat23} the author showed that $\lbrace\F\rbrace_{a\in\mathcal{K}}$ lies in the gap between modularity and non weak-modularity, making mating correspondences interesting from an measure theoretic view point. We focus on the family of compositions of two covering correspondences $\mathcal{F}_{R,S}=\covs\circ\covr$, where $R$ and $S$ are rational maps on $\widehat{\mathbb{C}}$. We show that $\lbrace\F\rbrace_{a\in\widehat{\mathbb{C}}\setminus\lbrace 1\rbrace}\subset\lbrace \mathcal{F}_{R,S}\rbrace_{R,S}$. In \cite{Bul00}, Bullett showed that the correspondences $\mathcal{F}_{R,S}$ for $(R,S)\in\mathcal{K}'$ are not far from being a mating between rational maps and Kleinian groups. The collections $\mathcal{K}$ and $\mathcal{K'}$ mentioned above are defined in Section \ref{com}. We note, however, that the main relevant property of the correspondences we consider is the existence of limit sets where the correspondence behaves like a rational map.

\begin{theorem}\label{B}
For any two rational maps $R$ and $S$ of degree at least $2$, the correspondence $\mathcal{F}_{R,S}$ is weakly-modular. Furthermore, if $(R,S)\in\mathcal{K}'$, $\mathcal{F}_{R,S}$ is not modular.
\end{theorem} 
In \cite{Mat23}, the author proved that for every $a\in\mathcal{K}$, there exist Borel probability measures $\mu_-$ and $\mu_+$ so that for all but at most 2 values of $z_0\in\mathbb{C}$,
$$\frac{1}{2^n}(\F^{-n})^*\delta_{z_0}\to\mu_-\hspace{1cm}\mbox{ and }\hspace{1cm}\frac{1}{2^n}(\F^n)^*\delta_{z_0}\to\mu_+$$
weakly, as $n\to\infty$. It is also shown that periodic points equidistribute, both with and without multiplicity. Moreover, the measures $\mu_-$ and $\mu_+$ are constructed from the measure of maximal entropy of Freire,Lopez and Ma\~n\'e in \cite{FreLopMan83}, and Lyubich in \cite{Lju83} for the quadratic map $P_A$. Dinh and Sibony prove Gromov's inequality for holomorphic correspondences in \cite{DinSib08II}, and Vivas and Sirvent prove a Half-Variational Principle for metric entropy in \cite{VivSir22}. A natural question is when can we find a \emph{measure of maximal entropy} (MME), meaning a measure that makes metric and topological entropies coincide, and further, when they equal $\log\max d$ for $d$-to-$d$ correspondences.

\begin{theorem}\label{A}
The measures $\mu_-$ and $\mu_+$ maximize entropy for $\F$ and $\F^{-1}$, respectively, with entropy $\log 2$.
\end{theorem}

\begin{theorem}\label{C}
Let $(R,S)\in\mathcal{K}'$ and suppose both $\mathcal{F}_{R,S}(\interior(\Delta_S))$ and $\mathcal{F}_{R,S}^{-1}(\interior(\Delta_{R}))$ are topological disks. Then there exists a measure $\mu_{R,S}$, supported on $\partial\Lambda_-$, that is $\mathcal{F}_{R,S}$-invariant and it maximizes entropy.
\end{theorem}
The following is an immediate consequence.
\begin{corollary} $h_{top}(\mathcal{F}_{R,S})=\log((\deg(R)-1)(\deg(S)-1))$.
\end{corollary}

The organization of the paper is as follows. In Section \ref{hol}, we introduce the notions of Holomorphic Correspondences and their actions on measures. We define as well deleted covering correspondences and prove basic properties about them. In section \ref{ent}, we prove basic properties of topological and metric entropy for correspondences. In Section \ref{com}, we define the composition of correspondences and we introduce the families  $\lbrace\F\rbrace_{a\in\mathcal{K}}$ and $\lbrace\mathcal{F}_{R,S}\rbrace_{(R,S)\in\mathcal{K}'}$ of correspondences on $\widehat{\mathbb{C}}$. In this section we define as well modularity and weak-modularity, and we prove Theorem \ref{B}. In section \ref{ent} we define limit sets of the above correspondences, and recall the rational behavior of the families in the limit sets. We finish this section by proving Theorem \ref{A} and Theorem \ref{C}.

\section{Holomorphic Correspondences}\label{hol}
Let $X$ be a compact Riemann surface. A \emph{holomorphic $1$-chain} in $X\times X$ is a formal sum $\sum_{j=1}^N n_j\Gamma_j$ of distinct irreducible complex subvarieties $\Gamma_j$ of dimension $1$, where $n_j\in\mathbb{Z}^+$. Let $\pi_i:X\times X\to X$ be the projection to the $i$-th coordinate, $i=1,2$. 

Suppose that $\restr{\pi_i}{\Gamma_j}$ is surjective, for $i=1,2$ and $j=1,\cdots, N$, and that $\pi_1^{-1}(z)\cap \Gamma_j$ is finite for all $z\in X$ and all $1\leq j\leq N$. The \emph{holomorphic correspondence $F$ given by $\Gamma$} is the multivalued map
$$F(z)\coloneqq\bigcup\limits_{j=1}^N\pi_2\left(\pi_1^{-1}(z)\cap\Gamma_j\right).$$
We say $\Gamma$ is the graph of the holomorphic correspondence $F$. The projections $\restr{\pi_i}{\Gamma_j}$ are proper, for $1\leq i\leq 2$, $1\leq j\leq N$. The correspondence is said to be $(d_1:d_2)$, where $d_1=\sum\limits_{j=1}^Nn_j\deg(\pi_2)$ and $d_2=\sum\limits_{j=1}^Nn_j\deg(\pi_1)$. That is, $d_1$ is the number of images of a point, and $d_2$ is the number of pre-images of a point, both counted with multiplicity. 

\begin{remark} Since $X$ is compact, then $\pi_2:X\times X\to X$ is closed. Moreover, $\pi_1$ is continuous and $\Gamma\subset X\times X$ is closed. Hence, for ever closed set $C\subset X$, we have that
$$F(C)=\pi_2(\pi_1^{-1}(C)\cap\Gamma)$$
is closed.
\end{remark}


Let $F_1$ and $F_2$ be two holomorphic correspondences on a compact Riemann surface $X$, whose graphs are $\Gamma_{F_1}$ and $\Gamma_{F_2}$, respectively. Let $\Gamma_{F_1\circ F_2}$ be the projection of $(\Gamma_{F_2}\times\Gamma_{F_1})\cap\lbrace z_2=z_3\rbrace$, to the first and fourth coordinates (counting multiplicities). The composition of $F_1$ and $F_2$ is the holomorphic correspondence $F_1\circ F_2$ whose graph is $\Gamma_{F_1\circ F_2}$.
\\
In terms of multivalued maps,
$$F_1\circ F_2(z)=\bigcup_{w\in F_2(z)}F_1(w).$$

\begin{definition} For a rational map $R:\widehat{\mathbb{C}}\to\widehat{\mathbb{C}}$ of degree $\deg{R}=r$, we define its \emph{deleted covering correspondence} $\covr$ on $\widehat{\mathbb{C}}$ to be the $(r-1:r-1)$ correspondence whose graph is the curve
$$\Gamma_R\coloneqq\overline{\lbrace (z,w)\in(\widehat{\mathbb{C}}\times\widehat{\mathbb{C}})\smallsetminus\mathfrak{D}_{\widehat{\mathbb{C}}}: R(z)=R(w)\rbrace},$$
where $\mathfrak{D}_{\widehat{\mathbb{C}}}\coloneqq\lbrace (z,z):z\in\widehat{\mathbb{C}}\rbrace$ is the diagonal in $\widehat{\mathbb{C}}\times\widehat{\mathbb{C}}$, and the closure is taken in $\widehat{\mathbb{C}}\times\widehat{\mathbb{C}}$.
\end{definition}

\begin{remark} It is clear from the definition that $\Gamma_R$ is symmetric with respect to the diagonal $\mathfrak{D}_{\widehat{\mathbb{C}}}=\lbrace(z,z)|z\in\widehat{\mathbb{C}}\rbrace$ of $\widehat{\mathbb{C}}\times\widehat{\mathbb{C}}$.

Another useful way to think of $\Gamma_R$ is as the zero set of polynomial in two variables
$$P(z,w)=\frac{R(z)-R(w)}{z-w}.$$
That is, $\covr$ relates pre-images under $R$, deleting the obvious relation of $z$ with itself. Note that it is still possible for some $(z,z)$ to belong to $\Gamma$.
\end{remark}



\begin{proposition}\label{inv} 
On $\widehat{\mathbb{C}}$, every involution other than the identity is the deleted covering correspondence of some quadratic rational map, and vice versa. Moreover, the involution fixes infinity if and only if the quadratic map is a polynomial.
\end{proposition}
\begin{proof}
Put $R(z)=\frac{az^2+bz+c}{dz^2+ez+f}$ of degree $2$, and note that for $z,w\notin R^{-1}(\infty)$, $z\neq w$,
$$\frac{R(z)-R(w)}{z-w}=\frac{(ae-bd)zw+(af-cd)(z+w)+(bf-ce)}{(dz^2+ez+f)(dw^2+ew+f)}.$$
Since neither $(dz^2+ez+f)$ or $(dw^2+ew+f)$ vanish,
$$\covr(z)=\overline{\lbrace w|(ae-bd)zw+(af-cd)(z+w)+(bf-ce)=0\rbrace}.$$
This yields that $\covr(z)=\frac{(cd-af)z+(ce-bf)}{(ae-bd)z-(cd-af)}$. This equality extends to the closure. Therefore, $\covr$ is an involution different from the identity.

On the other hand, for $J(z)=\frac{Az+B}{Cz-A}$ be an involution other than the identity. Making $a,b,c,d,e,f$ be so that $cd-af=A$, $ce-bf=B$ and $ae-bd=C$ we get that $J=\covr$ for $R(z)$ as above.
For instance, if $A\neq 0$, we take $a=1, b=\frac{B}{A},c=0,d=1,e=C+\frac{B}{A},f=-A$, and if $A=0$ we take $a=0, b=1, c=0, d=-C, e=0,f=-B$. 

Now observe that for a polynomial $P(z)=az^2+bz+c$ with $a\neq 0$, we have that $\frac{P(z)-P(w)}{z-w}
=a(z+w)+b,$
which implies that
$$\covp(z)=\lbrace w : az+aw+b=0\rbrace=\left\lbrace\frac{1}{a}(-az-b)\right\rbrace=\left\lbrace -z-\frac{b}{a}\right\rbrace.$$
By the symmetry of $\covp$, it is an involution. Moreover, $\covp$ clearly fixes $\infty$. If $J(z)=\frac{Az+B}{Cz-A}$ be an involution other than the identity that fixes infinity. Then $C=0$ and
$$J(z)=-z+\frac{B}{A}.$$
Taking $P(z)=Az^2+Bz$ we get that $\covp(z)=J(z)$.
\end{proof}

\begin{definition}\label{crit} 
Let $\Gamma=\sum\limits_{i=1}^N n_i\Gamma(i)$ be the graph of a holomorphic correspondence $F$ on a compact Riemann surface $X$. We denote
$$A_j(\Gamma)\coloneqq\bigcup_iA_j(\Gamma(i)),$$
where each $A_j(\Gamma(i))$ is defined to be
$$A_j(\Gamma)\coloneqq\left\lbrace\alpha\in\Gamma:\mbox{ for all open neighborhoods  }W\mbox{ of }\alpha,\restr{\pi_j}{W\cap\Gamma}\mbox{ is not injective}\right\rbrace.$$
We call $A_2(\Gamma)$ (respectively, $A_1(\Gamma)$) the set of \emph{ramification points} of $F$ (respectively, of $F^{-1}$). We define as well $B_j(\Gamma)=\pi_j(A_j(\Gamma))$, and we call $B_2(\Gamma)$ (respectively, $B_1(\Gamma)$) the set of \emph{critical values} of $F$ (respectively, $F^{-1}$).
\end{definition}

Denote by $\critpt(R)$ the set of critical points of the rational map $R$. We have the following.
\begin{proposition}\label{critical} 
Let $R$ be a rational map on $\widehat{\mathbb{C}}$ with degree $\deg(R)\geq 2$, and let $\Gamma_R$ be the graph of the deleted covering correspondence associated to $R$. Then $\covr$ sends open sets to open sets.
\end{proposition}
\begin{proof}
Using implicit differentiation on the equation $\frac{R(z)-R(w)}{z-w}=0$ we get that
$$\frac{dw}{dz}=\frac{(R(z)-R(w))-R'(z)(z-w)}{(R(z)-R(w))-R'(w)(z-w)}=\frac{R'(z)}{R'(w)}.$$

Let $U\subset\widehat{\mathbb{C}}$ be open and put $(z_0,w_0)\in\pi_1^{-1}(U)\cap\Gamma_R$. We will show that $w_0\in\interior(\covr(U))$. The case where $(z_0,w_0)\notin A_1(\Gamma_R)\cap A_2(\Gamma_R)$ is identical to the case where $R(z)=z^3-3z$ (see \cite[Remark 2.6]{Mat23}). Now let $(z_0,w_0)$ belong to $A_1(\Gamma_R)\cap A_2(\Gamma_R)$. Then both $z_0$ and $w_0$ are critical points of $R$. Let $f(z,w)=\frac{R(z)-R(w)}{z-w}$ and note that
$$\frac{\partial f}{\partial z}=\frac{(R'(z)-R'(w)\frac{dw}{dz})(z-w)-(R(z)-R(w))(1-\frac{dw}{dz})}{(z-w)^2}=-\frac{R(z)-R(w)}{(z-w)^2}.$$
Therefore $\frac{\partial f}{\partial z}(z_0,w_0)=0$, and similarly, $\frac{\partial f}{\partial w}(z_0,w_0)=0$. By the definition of $\Gamma_R$, we have that $(z_0,w_0)$ is a singular point of $\Gamma_R$.

By composing with a translation $t$, we can suppose that the singular point is the origin replacing $R$ by $R_\tau$. We blow-up $\mathbb{P}^1\times\mathbb{P}^1$ at the singularity $(0,0)$ by taking
$$\blow_{(0,0)}(\mathbb{P}^1\times\mathbb{P}^1)\coloneqq\lbrace (z,w,(Z:W))\in (\mathbb{P}^1\times\mathbb{P}^1)\times\mathbb{P}^1 |zW=wZ\rbrace$$
and $\beta:\blow_{(0,0)}(\mathbb{P}^1 \times\mathbb{P}^1)\to\mathbb{P}^1\times\mathbb{P}^1$ be the projection given by $(z,w,(Z:W))\mapsto (z,w)$. 

Put $\widetilde{\Gamma_{R_\tau}}\coloneqq\overline{\beta^{-1}(\Gamma_{R_\tau}\smallsetminus\lbrace (0,0)\rbrace)}$ and note that $\beta^{-1}(\Gamma_{R_\tau}\smallsetminus\lbrace(0,0)\rbrace)=\widetilde{\Gamma_{R_\tau}}\smallsetminus\beta^{-1}(0,0)$.


Let $p_j:(\mathbb{P}^1\times\mathbb{P}^1)\times\mathbb{P}^1\to\mathbb{P}^1$ be the projection to the $j$-th coordinate for $j=1,2$, and let $U\subset\mathbb{P}^1$ be an open neighborhood of $z=0$. We have that $\pi_1\circ\beta=p_1$ and
$$\beta^{-1}(\pi_1^{-1}(U)\cap(\Gamma_{R_\tau}\smallsetminus\lbrace(0,0)\rbrace))=\pi_1^{-1}(U)\cap(\widetilde{\Gamma_{R_\tau}}\smallsetminus\beta^{-1}(0,0)).$$

Similarly, we have that $p_2=\pi_2\circ\beta$. We proceed to show that $\pi_2(\pi_1^{-1}(U)\cap\Gamma_{R_\tau})$ is open. Indeed, 

\begin{eqnarray*}
\pi_2(\pi_1^{-1}(U)\cap\Gamma_{R_\tau})&=& p_2\circ\beta^{-1}(\beta(p_1^{-1}(U)\cap (\widetilde{\Gamma_{R_\tau}}\smallsetminus\beta^{-1}(0,0)))\cup (\beta\circ p_1^{-1}(0)\cap\widetilde{\Gamma_{R_\tau}}))\\
&=&p_2((p_1^{-1}(U)\cap(\widetilde{\Gamma_{R_\tau}}\smallsetminus\beta^{-1}(0,0))\cup\beta^{-1}(0,0)))\\
&=&p_2(p_1^{-1}(U)\cap\widetilde{\Gamma_{R_\tau}}).
\end{eqnarray*}

After finitely many blow-ups, the proper transform of the curve is non-singular. Suppose without of generality that it happens after the first blow-up, meaning that $\widetilde{\Gamma_{R_\tau}}$ is non-singular. Observe that the translation does not affect whether or not a point is singular. Let 
$$f_1(z,w,(Z:W))\coloneqq zW-Zw\hspace{1cm}\mbox{ and }\hspace{1cm}f_2(z,w,(Z:W))\coloneqq\frac{R_\tau(z)-R_\tau(w)}{z-w}.$$

Let $u$ be an affine coordinate of $(Z:W)$ on a neighborhood of a point $(0,0,u(0,0))\in\beta^{-1}(0,0)\cap\widetilde{\Gamma_{R_\tau}}$. There are no solutions to 
$$\frac{\partial f_1}{\partial z}=\frac{\partial f_1}{\partial w}=\frac{\partial f_1}{\partial u}=0.$$ 
Since $\widetilde{\Gamma_{R_\tau}}$ is non-singular and $\frac{\partial f_2}{\partial u}=0$, this implies that either $\frac{\partial f_2}{\partial z}\neq 0$ or $\frac{\partial f_2}{\partial w}\neq 0$.

We prove by contrapositive that this implies that either $\frac{dw}{dz}\neq 0$ or $\frac{dz}{dw}\neq 0$. Indeed, note that
$$\frac{\partial f_2}{\partial z}=\frac{(R'_\tau(z)-R'_\tau(w)\frac{dw}{dz})(z-w)-(R_\tau(z)-R_\tau(w))(1-\frac{dw}{dz})}{(z-w)^2}.$$
From Proposition \ref{critical} we have that if $\frac{dw}{dz}=\frac{dz}{dw}=0$, then $R'_\tau(z)=R'_\tau(w)=0$. Thus,
$$\frac{\partial f_2}{\partial z}=-\frac{R_\tau(z)-R_\tau(w)}{(z-w)^2}=0\hspace{1cm}\mbox{ and }\hspace{1cm}\frac{\partial f_2}{\partial z}=\frac{R_\tau(z)-R_\tau(w)}{(z-w)^2}=0.$$
Suppose without loss of generality that $\frac{dw}{dz}\neq 0$. The Implicit Function Theorem gives the existence of a neighborhood $\Omega$ of $0$ and a holomorphic function $g:\Omega\to\widehat{\mathbb{C}}$ so that the graph $\widetilde{\Gamma_{R_\tau}}$ is locally given by $\lbrace (z,w,u(z,w)):w=g(z)\rbrace$ at $(0,0,u(0,0))$. Since $\Gamma_{R_\tau}$ contains no lines, then $g$ is not constant, and hence it is open. Therefore, $(z_0,w_0,u(z_0,w_0))\in\interior(p_2(p_1^{-1}(U)\cap\widetilde{\Gamma_{R_\tau}}))$. We have that $\pi_2(\pi_1^{-1}(U)\cap\Gamma_{R_\tau})$ is open, as desired, and then so is $\covr$.
\end{proof}

\section{Entropy}\label{ent}
\subsection{Topological Entropy}

Let $F:X\to 2^X$ be a multivalued map on a compact metric space $(X,\operatorname{dist})$. For $x\in X$ and $n\in\mathbb{N}$, we define
$$\orb_n(F)=\lbrace (x_0,x_1,\cdots,x_{n-1})\in X^n :x_i\in F(x_{i-1})\mbox{ for all }1\leq i\leq n-1\rbrace.$$
We say that $A\subset\orb_n(F)$ is \emph{$\epsilon$-separated in the sense of Kelly-Tennant} if for every two distinct map $n$-orbits $(x_0,x_1,\cdots,x_{n-1}),(y_0,y_1,\cdots,y_{n-1})\in A$, there exists $0\leq i\leq n-1$ so that $\operatorname{dist}(x_i,y_i)\geq\epsilon$. Given $n\in\mathbb{N}$ and $\epsilon>0$, we define $s_{n,\epsilon}(F)$ to be the largest cardinality of an $\epsilon$-separated set collection of map $n$-orbits.
In \cite{KelTen17}, they defined \emph{map topological entropy} of the multivalued map $F$ as the quantity
$$h_{top}(F)\coloneqq\lim\limits_{\epsilon\to 0}\limsup\limits_{n\to\infty}\frac{1}{n}\log(s_{n,\epsilon}(F)).$$
The multivalued map $F$ is said to be  \emph{surjective} if for all $y\in X$, there exists $x\in X$ so that $y\in F(x)$. They prove that whenever $F$ is surjective, we have that $h_{top}(F^{-1})=h_{top}(F)$, where $F^{-1}$ is the multivalued map assigning $F^{-1}(y)=\lbrace x\in X: y\in F(x)\rbrace$.

In the case of a holomorphic correspondence $F$ with graph $\Gamma$ on a compact Riemann surface $X$, Dinh-Sibony \cite{DinSib08II} defined topological entropy as follows. Write $\Gamma=\sum_{j=1}^N\Gamma_j$ for the graph of $F$, where the $\Gamma_j 's$ are allowed to be repeated. For each $n\in\mathbb{N}$, a correspondence $n$-orbit is a $(2n+1)$-tuple
$$(x_0,x_1,\cdots,x_n;j_1,j_2,\cdots,j_n)\in X^{n+1}\times\lbrace 1,\cdots,N\rbrace^n$$
satisfying that for each $1\leq i\leq n$, $(x_{i-1},x_{i})\in\Gamma_{j_i}$. Put $\epsilon >0$. We say that a collection $\mathcal{A}$ of correspondence $n$-orbits is \emph{$\epsilon$-separated in the sense of Dinh-Sibony} if for all 
$$(x_0,x_1,\cdots,x_n;j_1,j_2,\cdots,j_n), (y_0,y_1,\cdots,y_n;k_1,k_2,\cdots,k_n)\in\mathcal{A},$$
either $\operatorname{dist}(x_i,y_i)>\epsilon\mbox{ for some }0\leq i\leq n$, or $j_i\neq k_i$ for some $1\leq i\leq n$, where $\operatorname{dist}(\cdot,\cdot)$ denotes a distance on $X$. We denote $$N_{\Gamma}(\epsilon,n)=\max\lbrace|\mathcal{A}|: \mathcal{A}\mbox{ is an }\epsilon\mbox{-separated collection of corresp. } n\mbox{-orbits}\rbrace.$$

\begin{definition} The \emph{correspondence topological entropy} of $\Gamma$ is defined to be the number
$$h(\Gamma)\coloneqq\sup\limits_{\epsilon >0}\limsup\limits_{n\to\infty}\frac{1}{n}\log N_{\Gamma}(\epsilon,n).$$
\end{definition}

By compactness of $X$, every $\epsilon$-separated family of correspondence $n$-orbits is finite. Moreover, $h(\Gamma)$ does not depend on the choice of $\operatorname{dist}(\cdot,\cdot)$.


The following result is from \cite{DinSib08II}, generalizing Gromov's result in \cite{Gro03}.

\begin{theorem}[Gromov's Inequality]\label{bound} 
Let $X$ be a compact Riemann surface and let $\Gamma$ be the graph of a holomorphic correspondence $F$ on $X$. Then
$$h(\Gamma)\leq\log\max\lbrace d_1(F),d_2(F)\rbrace.$$
\end{theorem}

Observe that both $h_{top}$ and $h$ generalize the topological entropy of single-valued maps. Furthermore, holomorphic correspondences always induce multivalued maps, and we have the following lemma.

\begin{lemma}\label{entropy inequality}
Let $X$ be a compact Riemann surface and let $F$ be a holomorphic correspondence on $X$ with graph $\Gamma$. Then $h_{top}(F)\leq h(\Gamma)$.
\end{lemma}
\begin{proof}
Put $\Gamma=\sum_{i=1}^m n_i\Gamma_i$. Note that if $(x_0,x_1,\cdots,x_{n-1})\in\orb_n(F)$, then there exists $(j_1,\cdots,j_{n-1})\in\lbrace 1,\cdots, m\rbrace^{n-1}$ so that 
$$(x_0,x_1,\cdots,x_{n-1};j_1,\cdots,j_{n-1})$$
satisfies that for each $1\leq i\leq n-1$, $(x_{i-1},x_i)\in\Gamma_{j_i}$. Thus, if $A\subset\orb_n(F)$ is $\epsilon$-separated, then
$$\mathcal{A}=\lbrace (x_0,x_1,\cdots,x_{n-1};j_1,\cdots,j_{n-1}):(x_0,x_1,\cdots,x_{n-1})\in A\rbrace$$
obtained this way, is a $\delta$-separated collection of correspondence $(n-1)$-orbits, for all $\delta<\epsilon$. Therefore, $s_{n,\epsilon}(F)\leq N_{\Gamma}(n-1,\delta)$, for all $\delta<\epsilon$. Now observe that $s_{n,\delta}$ is decreasing with $\delta$. Taking limit as $\delta\to\epsilon$, we get that 
$$s_{n,\epsilon}(F)\leq N_{\Gamma}(n-1,\epsilon)$$. Moreover,
\begin{eqnarray*}
h_{top}&=&\lim\limits_{\epsilon\to 0}\limsup\limits_{n\to\infty}\frac{1}{n}\log(s_{n,\epsilon}(F))\\
&=&\sup\limits_{\epsilon >0}\limsup_{n\to\infty}\frac{1}{n}\log(s_{n,\epsilon}(F))\\
&\leq&\sup\limits_{\epsilon >0}\limsup\limits_{n\to\infty}\frac{1}{n}\log(N_{\Gamma}(n-1,\epsilon))\\
&=&\sup\limits_{\epsilon >0}\limsup\limits_{n\to\infty}\frac{n-1}{n}\frac{1}{n-1}\log(N_{\Gamma}(n-1,\epsilon))\\
&=&h(\Gamma_a).
\end{eqnarray*}
\end{proof}

Dinh-Sibony's proof of Gromov's inequality works for meromorphic correspondences on compact K\"ahler manifolds, and so does the proof of this Lemma. For this paper, we restrict ourselves to Riemann surfaces since our main motivation is compositions of covering correspondences that act on the Riemann sphere.

\subsection{Metric Entropy} We proceed to define the metric entropy in the sense of \cite{VivSir22}, and state their Half-Variational Principle. 

Let $X$ be a compact metric space. Let $F$ be a multivalued map on $X$ that satisfies that $F(C)$ is closed for every closed set $C$ and let $\mu$ be a Borel probability measure on $X$. We say that $\mu$ is \emph{$F$-invariant} if $\mu(F^{-1}(A))\geq\mu(A)$, for all $A\in\mathcal{B}$. We denote by $\mathcal{M}_1(F)$ the collection of all such measures. For a finite ordered partition $\mathcal{P}=\lbrace P_1,\cdots,P_k\rbrace$ of $X$, the \emph{entropy of $\mathcal{P}$} is defined by
$$H_\mu(\mathcal{P})\coloneqq\sum\limits_{P\in\mathcal{P}}-\mu(P)\log\mu(P).$$

Now, let $\widetilde{\mathcal{P}}_n$ be the finite ordered partition given by
$$P_{1,n}\coloneqq F^{-n}(P_1),\mbox{ and }P_{j,n}\coloneqq F^{-n}(P_j)\setminus\bigcup\limits_{\ell<j}F^{-n}(P_{\ell}),\hspace{.5cm}2\leq j\leq k.$$
The \emph{metric entropy of the system $(F,\mu)$} is defined as
$$h_\mu(F)\coloneqq\sup\limits_{\mathcal{P}}\limsup_{N\to\infty}\frac{1}{N}H_{\mu}\left(\bigvee\limits_{n=0}^{N-1}\widetilde{\mathcal{P}}_n\right)$$
where the supremum is taken over all the finite totally ordered partitions. This definition restricts to the usual notion of metric entropy for maps. We also have several of the properties that entropies for maps have. For instance, let $F$ be so that $F(x)$ is closed and nonempty for all $x$, then for all $k\geq 1$ and all $\mu\in\mathcal{M}_1(F)$, we have that
$$h_{\mu}(F^k)\geq k h_{\mu}(F).$$
There is also a Half-Variational Principle, which is the main result in \cite{VivSir22}.

\begin{theorem}[Half-Variational Principle]
Let $X$ be a compact metric space $X$ and let $F:X\to 2^X$ be a multivalued map on $X$ with the property that $F(C)\subset X$ is closed whenever $C\subset X$ is closed. Then,
$$h_{top}(F)\geq\sup\limits_{\mu\in\mathcal{M}_1(F)}h_\mu(F).$$
\end{theorem}

Observe that holomorphic correspondences on compact Riemann Surfaces always satisfy the conditions in the above theorem.

Let $(S,\mathcal{A},\mu)$ be a measure space with $S\subset X$. Put $\hat{\mathcal{A}}=\lbrace A\subset X|A\cap S\in\mathcal{A}\rbrace$ and define $\hat{\mu}$ by
$$\hat{\mu}(A)=\mu(A\cap S).$$
Then $(X,\hat{\mathcal{A}},\hat{\mu})$ is a measure space as well. Furthermore, $\supp(\hat{\mu})=\supp(\mu)$. With this notation, we have the following lemma.

\begin{lemma}\label{restriction}
Let $F$ be a multivalued map on a compact metric space $X$. Suppose there exists a set $S\subset X$ satisfying $F^{-1}(S)= S$, for which the two-sided restriction $F|:S\to S$ given by $F|(z)=F(z)\cap S$ for $z\in S$ is a measurable function. Let $\hat{\mu}$ as above. Then, for all $\mu\in\mathcal{M}_1(F|)$, we have that
$$h_{\mu}(F|)\leq h_{\hat{\mu}}(F).$$
\end{lemma}

\begin{proof}
Observe that $\hat{\mu}$ is invariant for $F$ in the sense of multivalued maps, since
$$\hat{\mu}(F^{-1}(A))=\mu(F^{-1}(A)\cap S)=\mu(F^{-1}(A)\cap F^{-1}(S))=\mu(F^{-1}(A\cap S))\geq\mu(A\cap S)=\hat{\mu}(A).$$
Let $\mathcal{P}=\lbrace P_1,P_2,\cdots,P_k\rbrace$ be a finite ordered measurable partition of $S$ with respect to the measure $\mu$, and write $\mathcal{P}^\dagger=\mathcal{P}\cup\lbrace X\setminus S\rbrace$. Then $\mathcal{P}^\dagger$ is a finite ordered measurable partition of $(X,\hat{\mu})$. Moreover,
\begin{eqnarray*}
\bigcup\limits_{j=1}^k F^{-1}(P_j)&=&\lbrace z|F(z)\cap P_j\neq\emptyset,\mbox{ for some }j\rbrace\\
&=&\left\lbrace z\Big|F(z)\cap\left(\bigcup\limits_{j=1}^k P_j\right)\neq\emptyset\right\rbrace\\
&=&F^{-1}\left(\bigcup\limits_{j=1}^k P_j\right)\\
&=&F^{-1}(S)\\
&=&S.
\end{eqnarray*}

Therefore,
\begin{eqnarray*}
\widetilde{(\mathcal{P}^\dagger)}_1&=&\left\lbrace F^{-1}(P_1),F^{-1}(P_2)\setminus F^{-1}(P_1),\cdots, F^{-1}(P_k)\setminus\bigcup\limits_{j=1}^{k-1} F^{-1}(P_j), F^{-1}(X\setminus S)\setminus S\right\rbrace\\
&=&\left\lbrace F^{-1}(P_1),F^{-1}(P_2)\setminus F^{-1}(P_1),\cdots,F^{-1}(P_k)\setminus\bigcup\limits_{j=1}^{k-1} F^{-1}(P_j),X\setminus S\right\rbrace\\
&=&(\widetilde{\mathcal{P}_1})^\dagger.
\end{eqnarray*}

Recursively, $\widetilde{(\mathcal{P}^\dagger)}_n=(\widetilde{\mathcal{P}}_n)^\dagger$, for all $n\geq 0$. In addition, we also have that
$$\mathcal{P}^\dagger\vee\mathcal{Q}^\dagger=\left(\mathcal{P}\vee\mathcal{Q}\right)\cup\lbrace X\setminus S\rbrace=(\mathcal{P}\vee\mathcal{Q})^\dagger,$$
for all $n,m\geq 0$. Thus, since $\supp(\hat{\mu})\subset S$,
\begin{eqnarray*}
H_{\hat{\mu}}\left(F,\bigvee\limits_{n=0}^{N-1}(\widetilde{\mathcal{P}}_n)^\dagger\right)&=&H_{\hat{\mu}}\left(F,\left(\bigvee\limits_{n=0}^{N-1}\widetilde{\mathcal{P}}_n\right)\cup\lbrace X\setminus S\rbrace\right)\\
&=&H_{\mu}\left(F|,\bigvee\limits_{n=0}^{N-1}\widetilde{\mathcal{P}}_n\right).
\end{eqnarray*}

Therefore,
$$h_{\hat{\mu}}\left(F,\mathcal{P}^\dagger\right)=\limsup\limits_{N\to\infty}\frac{1}{N}H_{\hat{\mu}}\left(F,\bigvee\limits_{n=0}^{N-1}\widetilde{(\mathcal{P}^\dagger)}_n\right)=\limsup\limits_{N\to\infty}\frac{1}{N}H_{\mu}\left(F|,\bigvee\limits_{n=0}^{N-1}\widetilde{\mathcal{P}}_n\right)=h_{\mu}\left(F|,\mathcal{P}\right).$$

From the construction of the partition $\mathcal{P}^\dagger$ from the partition $\mathcal{P}$, after taking supremum over all partitions $\mathcal{P}$ of $S$ we get that
$$h_{\mu}(F|)\leq h_{\hat{\mu}}(F).$$
\end{proof}

Putting together results from \cite{CarMetMor15} and \cite{VivSir22}, we have the following proposition.

\begin{proposition}\label{entropy0} 
Let $F$ be a multivalued map on a compact metric space $(X,d)$ satisfying:
\begin{itemize}
\item[(i)] For all $\varepsilon >0$, there exists $\delta >0$ so that for $x,y\in X$ with $d(x.y)<\delta$, there exist $(x_i)_{i=0}^{\infty}, (y_i)_{i=0}^{\infty}$ such that $x_0=x$, $y_0=y$, $x_{i+1}\in F(x_i)$, $y_{i+1}\in F(y_i)$ and $d(x_i,y_i)<\varepsilon$, for all $i\geq 0$, and
\item[(ii)] For all $\varepsilon>0$, there exists $\delta>0$ so that for all $x\in X$ and for all $n\in\mathbb{N}$, there exists $x_i\in F^i(x)$, $0\leq i\leq n-1$ such that $x_0=x$ and $B_n(x,\delta)\subseteq\bigcap_{i=0}^{n-1} F^{-i}(B(x_i,\epsilon))$.
\end{itemize}
Then, $$\sup\limits_{\mu\in\mathcal{M}_1(F)}h_{\mu}(F)=0.$$
\end{proposition}

\begin{lemma}\label{radius}
Let $R$ be a rational map. For all $\delta>0$ small enough, there exists $r_R(\delta)>0$ such that $\lim\limits_{\delta\to 0}r_R(\delta)=0$ and $\operatorname{diam}(U)\leq r_R(\delta)$, for each component $U$ of $R^{-1}(B(x,\delta))$, for all $x\in\widehat{\mathbb{C}}$.
\end{lemma}

\begin{proof}
Put $p\in\widehat{\mathbb{C}}$. By the Local Normal Form Theorem, there exists $(U(p),\Phi_{1,p},\Phi_{2,p},m(p))$ be so that 
\begin{itemize}
\item $U(p)$ is a topological disk containing $p$,
\item $\Phi_{1,p}(U(p))=\mathbb{D}$,
\item $\Phi_{2,p}(R(U(p))=\mathbb{D}$,
\item $\Phi_{2,p}\circ R\circ \Phi_{1,p}^{-1}(z)=z^{m(p)}$, for some $m(p)\in\mathbb{Z}$.
\end{itemize}
Since $\bigcup_{p\in\widehat{\mathbb{C}}}R(U(p))=\widehat{\mathbb{C}}$ and $\widehat{\mathbb{C}}$ is compact, there exist $p_1,\cdot,p_N\in\widehat{\mathbb{C}}$ satisfying that $\bigcup_{j=1}^N R(U(p_j))=\widehat{\mathbb{C}}$.
\\

Let $f(z)=z^m$. We know that for every $A\subset\mathbb{C}$, $f^{-1}(A)=\bigcup_{k=1}^{m}g_k(A)$, where $g_k(re^{i\theta})=\zeta^kr^{\frac{1}{m}}e^{i\frac{\theta}{m}}$ for $1\leq k\leq m$. 
Let $w\in\mathbb{D}$ and let $\epsilon>0$ be so that $B(z,2\epsilon)\subset\mathbb{D}$. 
\\

If $w=0$, then $g_k(B(w,\epsilon))=B(0,\delta^{\frac{1}{m}})$. Therefore, each component of $f^{-1}(B(w,\delta))$ has diameter at most $w\delta^{\frac{1}{m}}$.
\\

If $0\in B(w,\epsilon)$, then $B(w,\epsilon)\subset B(w,2\epsilon)\subset\mathbb{D}$. By above, each component of $f^{-1}(B(w,\epsilon))$ has diameter at most $2(w\epsilon)^{\frac{1}{m}}$.
\\

If $0\notin B(w,\epsilon)$ and $w'\in B(w,\epsilon)$, write $w'=|w'|e^{i\theta}$ with $\theta\in (\Arg(w)-\pi,\Arg(w)+\pi)$, where $\Arg(w)$ is the principal branch of logarithm. Then $|w'|\in(|w|-\epsilon,|w|+\epsilon)$ and $\theta\in(\Arg(w)-\epsilon,\Arg(w)+\epsilon)$. Thus,
\begin{eqnarray*}
|g_k(w')-g_k(w)|&=&\big| |w'|^{\frac{1}{m}}e^{i\frac{\theta}{m}}-|w|^{\frac{1}{m}}e^{i\frac{\theta}{m}}+|w|^{\frac{1}{m}}e^{i\frac{\theta}{m}}-|w|^{\frac{1}{m}}e^{i\frac{\Arg(z)}{m}} \big|\\
&=&\big| (|w'|^{\frac{1}{m}}-|w|^{\frac{1}{m}})e^{i\frac{\theta}{m}}+|w|^{\frac{1}{m}}(e^{i\frac{\theta}{m}}-e^{i\frac{\Arg(w)}{m}}) \big|\\
&<&\epsilon^{\frac{1}{m}}+\frac{\pi}{m}\epsilon.
\end{eqnarray*}
Therefore, each component of $f^{-1}(B(w,\epsilon))$ has diameter at most $\epsilon^{\frac{1}{m}}+\frac{\pi}{m}\epsilon$.

Observe that for each $1\leq j\leq N$, both $\Phi_{2,p_j}$ and $\Phi_{1,p_j}^{-1}$ are Lipschitz. Let $L$ be the maximum Lipschitz constant among all $\Phi_{2,p_j}$, $1\leq j\leq N$, and $K$ be the maximum Lipschitz constant among all $\Phi_{1,p_j}^{-1}$, $1\leq j\leq N$. Denote by $\delta_0$ the Lebesgue number of the covering $\lbrace R(U(p_j))\rbrace_{j=1}^N$ and let $\delta<\frac{\delta_0}{2}$. Then there exists $1\leq j\leq N$ so that $B(z,\delta)\subset R(U(p_j))$. Thus,
$$\Phi_{2,p_j}(B(z,\delta))\subset B(\Phi_{2,p}(z),L\delta)\subset\mathbb{D}.$$
Then, each component of $f^{-1}(\Phi_{2,p_j}(B(z,\delta)))$ has diameter at most $\max\lbrace 2(2L\delta)^{\frac{1}{m}},(L\delta)^{\frac{1}{m}}+\frac{\pi}{m}(L\delta)\rbrace$. Finally, since $R^{-1}(B(z,\delta))=\Phi_{1,p_j}^{-1}f^{-1}\Phi_{2,p_j}(B(z,\delta))$, then each component of $R^{-1}(B(z,\delta))$ has diameter at most
$$r_R(\delta)\coloneqq\max\left\lbrace 2K(2L\delta)^{\frac{1}{m}},K((L\delta)^{\frac{1}{m}}+\dfrac{\pi}{m}(L\delta))\right\rbrace.$$
From the definition of $r_R(\delta)$, it is clear that $\lim\limits_{\delta\to 0}r_R(\delta)=0$. 
\end{proof}

For the purpose of the proof of the following lemma, we use the notation $r_R^0(\delta)=\delta$, $r_R^n(\delta)=r_R(r_R^{n-1}(\delta))$, for $n\geq 1$.

\begin{lemma}\label{rational}
Let $R$ be a rational map on $\widehat{\mathbb{C}}$ of degree $\operatorname{deg}(R)\geq 2$. Then 
$$\sup\limits_{\mu}h_\mu(\restr{R^{-1}}{\mathcal{J}(R)})=0,$$
where the supremum runs over all $\mu\in\mathcal{M}_1\left(\restr{R^{-1}}{\mathcal{J}(R))}\right)$. Furthermore, if $R$ has empty exceptional set, then
$$\sup\limits_{\mu\in\mathcal{M}_1(R^{-1})}h_{\mu}(R^{-1})=0.$$
\end{lemma}

\begin{proof}
We prove the case where $R$ has empty exceptional set, as the proof works as well for the restriction to the Julia set. Let $p$ be a repelling periodic point of minimal period $n_0$ and multiplier $\lambda$, and put $\tilde{R}=R^{n_0}$, then $R^{n_0}$ also has empty exceptional set. We will show that $\tilde{R}$ satisfies properties (i) and (ii) of Proposition \ref{entropy0}.
\\

Let $\rho>0$ be so that $\restr{\tilde{R}}{B(p,\rho)}$ is conjugate to $z\mapsto \lambda z$ around the origin, where $|\lambda |>1$. Fix $0<\epsilon<\rho$. There exists $N\in\mathbb{N}$ such that $\tilde{R}^N(B(p,\frac{\epsilon}{2}))=\widehat{\mathbb{C}}$. Thus, for every $x\in\widehat{\mathbb{C}}$, there exists $x_N\in \tilde{R}^{-N}(x)\cap B(p,\frac{\epsilon}{2})$. Put $x_i=\tilde{R}^{N-i}(x_N)$ for $1\leq i\leq N$.
\\

Let $\delta>0$ be small so that $\max\lbrace r_{\tilde{R}}^i(\delta):0\leq i\leq N\rbrace<\frac{\epsilon}{2}$. Let $y\in B(x,\delta)$ and $U$ be the connected component of $\tilde{R}^{-1}(B(x,\delta))$ containing $x_1$. Since $\tilde{R}$ maps $U$ onto $B(x,\delta)$, there exists  $y_1\in \tilde{R}^{-1}(y)\cap U$, and we have that $d(x_1,y_1)< r_{\tilde{R}}(\delta)$. Recursively, we find $(y_i)_{i=1}^N$ so that $y_{i+1}\in \tilde{R}^{-1}(y_i)$ and $d(x_i,y_i)<r_{\tilde{R}}^i(\delta)<\frac{\epsilon}{2}$, for $0\leq i\leq N$.\\

Since $y_N\in B(x_N,\frac{\epsilon}{2})$, then $x_N,y_N\in B(p,\rho)$. Since on $B(p,\rho)$, $\tilde{R}$ is conjugate to $z\to\lambda z$ with $|\lambda|>1$, then for any $(x_i)_{i=N+1}^\infty, (y_i)_{i=N+1}^\infty$ so that $x_i\in \tilde{R}^{-1}(x_{i-1})$ and $y_i\in \tilde{R}^{-1}(y_{i-1})$ for all $i\geq N+1$, we have that $d(x_i,y_i)<\epsilon$ as well.
\\

Thus, making $x_0=x$ and $y_0=y$, we have the existence of sequences $(x_i)_{i=0}^\infty$, $(y_i)_{i=0}^\infty$ with $x_0=x$, $y_0=y$, $x_{i+1}\in\tilde{R}^{-1}(x_i)$, $y_{i+1}\in \tilde{R}^{-1}(y_i)$ and $d(x_i,y_i)<\epsilon$, for all $i\geq 0$. That is, $\tilde{R}$ satisfies property (i) of Proposition \ref{entropy0}.
\\

Now fix $\epsilon>0$, and note that $\lbrace B(\zeta,\frac{\epsilon}{2})\rbrace_{\zeta\in\mathcal{J}(\tilde{R})}$ is an open cover of $\mathcal{J}(\tilde{R})$. Since $\mathcal{J}(\tilde{R})$ is compact, there is a finite subcover $\lbrace B(\zeta_j,\frac{\epsilon}{2})\rbrace_{j=1}^n$. For each $1\leq j\leq n$, there exists $n_j\in\mathbb{N}$ such that $\tilde{R}^{n_j}(B(\zeta_j,\frac{\epsilon}{2}))=\widehat{\mathbb{C}}$. Put $K=\max\lbrace n_j:1\leq j\leq n\rbrace$. We know $R^k(B(\zeta_j,\frac{\epsilon}{2}))=\widehat{\mathbb{C}}$, for all $k\geq K$ and $1\leq j\leq n$. In particular, if $U=\bigcup_{j=1}^n B(\zeta_j,\frac{\epsilon}{2})$, then $R^k(U)=\widehat{\mathbb{C}}$.

For all $x\in\widehat{\mathbb{C}}$ and $i\geq K$, there exists $x_i\in \tilde{R}^{-i}(z)\cap U$. Thus, there exists $1\leq j\leq n$ such that $x_i\in B(\zeta_j,\frac{\epsilon}{2})$. Since $B(\zeta_j,\frac{\epsilon}{2})\subset B(x_i,\epsilon)$, then $\tilde{R}^i(B(x_i,\epsilon))=\widehat{\mathbb{C}}$. Thus, for all $n\geq K+1$,
\begin{equation}\label{intersection1}
\bigcap\limits_{i=K}^{n-1} \tilde{R}^i(B(x_i,\epsilon))=\widehat{\mathbb{C}}.\end{equation}

Let $\delta>0$ be such that $\max\lbrace r_{\tilde{R}}^j(\delta):0\leq j\leq K\rbrace<\epsilon$. Let $(x_i)_{i=0}^K$ be any sequence satisfying $x_0=x$ and $x_i\in \tilde{R}^{-i}(x)$, $1\leq i\leq K$. If $y\in B(x,\delta)$ and $U$ is the connected component of $\tilde{R}^{-1}(B(x,\delta))$ containing $x_1$, then there exists a unique $y_1\in U$ with $\tilde{R}(y_1)=y$. By Lemma \ref{radius}, $d(x_1,y_1)<r(\delta)<\epsilon$. Now let $V$ be the connected component of $\tilde{R}^{-1}(U)$ containing $x_2$. There exists $y_2\in V$ so that $\tilde{R}(y_2)=y_1$ and $d(x_2,y_2)<\epsilon$. Recursively, we find $(y_i)_{i=0}^K$ s.t. $y_0=y$, $y_i\in \tilde{R}^{-i}(y)$ and $d(x_i,y_i)<\epsilon$. Therefore, $y_i\in B(x_i,\epsilon)$, and $y=R^i(y_i)\in\tilde{R}^{-1}(B(x_i,\epsilon))$, for $0\leq i\leq K$. This proves that
$$B(x,\delta)\subset\bigcap\limits_{i=0}^{n-1}\tilde{R}^i(B(x_i,\epsilon)),$$
for all $n\leq K$. Furthermore, using equation (\ref{intersection1}) we get that
$$B(x,\delta)\subset\bigcap\limits_{i=0}^{K-1}\tilde{R}^i(B(x_i,\epsilon))=\bigcap\limits_{i=0}^{n-1}\tilde{R}^i(B(x_i,\epsilon)),$$
for all $n\geq K+1$. This proves that $\tilde{R}$ satisfies property (ii) of Proposition \ref{entropy0}. 

If $\mu\in\mathcal{M}_1(R^{-1})$, then $\mu\in\mathcal{M}_1(\tilde{R}^{-1})$ and $h_{\mu}(\tilde{R}^{-1})=0$. Therefore,
$$h_{\mu}(R^{-1})\leq\frac{1}{n_0}h_{\mu}(\tilde{R}^{-1})=0.$$
Since this is true for arbitrary $\mu\in\mathcal{M}_1(R^{-1})$, this proves the desired result.
\end{proof}

\section{Composition of Covering Correspondences}\label{com}
In this section we will define the families $\lbrace\F\rbrace_{a\in\mathcal{K}}$ and $\lbrace\mathcal{F}_{R,S}\rbrace_{(R,S)\in\mathcal{K}'}$ whose topological and metric entropy we compute in this work. In addition, we prove Theorem \ref{B}. In order to do so, we start by defining the composition of holomorphic correspondences.

For $a\neq 1$, we take the involution
$$\J(z)=\frac{(a+1)z-2a}{2z-(a+1)}$$
and define $\F$ to be the correspondence given by $\F\coloneqq\J\circ\covq$, where $Q(z)=z^3-3z$. Observe that $d_1(\F)=d_2(\F)=2$ and that $\F(1)=\lbrace 1\rbrace$ with multiplicity $2$. In the next section we will talk about the behavior around this fixed point, which is work by Bullett-Lomonaco \cite{BulLom20I}. We note that, from the previous section, we can write $\J(z)=\covr$, where $R(z)=\frac{z^2-a}{(z-1)^2}.$

We define
$$\mathcal{F}_{R,S}\coloneqq\covr\circ\covs$$
for $R,S$ rational maps of degree greater than or equal to $2$. It is immediate from Proposition \ref{inv} that $\lbrace\F\rbrace_{a\in\mathbb{C}\setminus\lbrace 1\rbrace}\subset\lbrace\mathcal{F}_{R,S}\rbrace_{R,S}$, where the latter runs over all rational maps $R$ and $S$.

\begin{definition}

A \emph{Klein combination pair} $(\Delta_{\J},\Delta_{\covq})$ for $\F$ is a pair of open subsets of $\widehat{\mathbb{C}}$ so that $\J(\Delta_{\J})\cap\Delta_{\J}=\emptyset$, $\covq(\Delta_{\covq})\cap\Delta_{\covq}=\emptyset,$ and 
$$\Delta_{\J}\cup\Delta_{\covq}=\widehat{\mathbb{C}}\setminus\lbrace 1\rbrace.$$
We denote by $\mathcal{K}$ the collection of $a\in\mathbb{C}\setminus\lbrace 1\rbrace$ for which there exists a Klein combination pair $(\Delta_{\J},\Delta_{\covq})$ for $\F$. 
\end{definition}

The restricted family $\lbrace \F\rbrace_{a\in\mathcal{K}}$ is described as a mating of parabolic quadratic rational maps and $\PSL2(\mathbb{Z})$ in  \cite{BulLom20I}.
\begin{definition}
A \emph{transversal} $\Delta_R\subset\widehat{\mathbb{C}}$ for $\covr$ is a maximal set where the rational map $R$ is injective.
\\
A \emph{Klein combination pair} $(\Delta_R,\Delta_S)$ of $\mathcal{F}_{R,S}$ is a pair of subsets of $\widehat{\mathbb{C}}$ so that:
\begin{itemize}
\item $\Delta_R$ and $\Delta_S$ are transversals for $R$ and $S$, respectively,
\item $\overline{\interior(\Delta_R)}=\overline{\Delta_R}$ and $\overline{\interior(\Delta_S)}=\overline{\Delta_S}$,
\item $\interior(\Delta_R)\cup\interior(\Delta_S)=\widehat{\mathbb{C}}$,
\item $\mathcal{F}_{R,S}(\Delta_S)$ and $\mathcal{F}_{R,S}^{-1}(\Delta_R)$ are topological disks.
\end{itemize}
We denote by $\mathcal{K}'$ the collection of $(R,S)$ for which there is a Klein combination pair for $\mathcal{F}_{R,S}$.
\end{definition}

In \cite{Bul00}, the author showed that under these conditions,
$$\covr(\interior(\Delta_R))\cap\overline{\Delta_S}=\emptyset.$$
Furthermore, he proved a Klein Combination theorem for the family $\lbrace\mathcal{F}_{R,S}\rbrace_{(R,S)\in\mathcal{K}'}$.

\begin{definition}
Let $G$ be a connected Lie group, $\Lambda$ be a torsion free lattice and $K$ a compact Lie subgroup of $G$. A modular correspondence $F$ on $X=\Lambda\setminus G/K$ is that whose graph has the form $\sum_i n_i\Gamma_{g_i}$, where each $\Gamma_{g_i}$ is the projection to $X\times X$ of the graph in $G\times G$ of left multiplication by $g_i\in G$ satisfying $[\Lambda:g_i\Lambda g_i^{-1}\cap\Lambda]<\infty$. 
\end{definition}

\begin{definition}
Let $F$ be a holomorphic correspondence on a compact Riemann surface, satisfying $d_1(F)=d_2(F)=d$. We say $F$ is \emph{weakly-modular} if there exists Borel probability measures $\mu_1$ and $\mu_2$ on $X$, with the property that $(\restr{\pi_1}{\Gamma})^*\mu_1=(\restr{\pi_2}{\Gamma})^*\mu_2$.
\end{definition}

Modular correspondences which are holomorphic on a compact Riemann surface, are always weakly-modular. Furthermore, every modular correspondence $F$ has a Borel probability measure $\lambda$ that assigns positive measure to open sets, and so that $F^*\lambda=d\lambda$.


Let $X$ be a compact Riemann surface and let $F$ be a holomorphic correspondence on $X$, whose graph is $\Gamma$. If $\alpha$ is a smooth form on $X$, we define
$$F^*(\alpha)\coloneqq{\pi_2}_*(\Gamma\wedge\pi_1^*(\alpha))\mbox{ and }F_*(\alpha)\coloneqq{\pi_1}_*(\Gamma\wedge\pi_2^*(\alpha))$$
to be the \emph{pull-back} and \emph{push-forward}, respectively, of $\alpha$ under $F$. We have that $F^*(\alpha)$ is a form that is smooth on $X\setminus B_1(\Gamma)$ and $F_*(\alpha)$ is smooth on $X\setminus B_2(\Gamma)$. 

We define
$$L^2_{(1,0)}\coloneqq\lbrace\phi: \phi\mbox{ is a }(1,0)-\mbox{form on }X\mbox{ with }L^2\mbox{ coefficients}\rbrace.$$
For $\phi\in L^2_{(1,0)}$, we define
$$\|\phi\|_{L^2}\coloneqq\left(\int_X i\phi\wedge\overline{\phi}\right)^{\frac{1}{2}}.$$
This is well defined since $i\phi\wedge\overline{\phi}$ is a $(1,1)$-form. 
We define as well $\|\cdot\|$ to be the operator norm.

In \cite{DinKauWu20}, Dinh-Kauffman-Wu proved the following proposition, which was key for their equidistribution result.

\begin{proposition}\label{norm}
Let $F$ be a holomorphic correspondence on a compact Riemann surface $X$ with $d(F)=d(F^{-1})=d$. Let $\Gamma$ be the graph of $F$. Then the pull-back action of $F$ on smooth $(1,0)$-form extends to an  operator on $L^2_{(1,0)}$, and
\begin{itemize}
\item[(a)] The operator $\frac{1}{d}F^*:L^2_{(1,0)}\rightarrow L^2_{(1,0)}$ satisfies $\Vert\frac{1}{d}F^*\Vert\leq 1$.
\item[(b)] $\|\frac{1}{d}F^*\phi\|_{L^2}=\|\phi\|_{L^2}$ for $\phi\in L^2_{(1,0)}$ if and only if for every $U\subset X\setminus B_1(\Gamma)$ and for every pair of local branches $f_1$ and $f_2$ of $F$ on $U$, the equality $f_1^*\phi=f_2^*\phi$ holds on $U$.
\item[(c)] If $\Vert\frac{1}{d}F^*\Vert= 1$, then $F$ is weakly-modular.
\end{itemize}
\end{proposition}

\begin{proof}[Proof of Theorem \ref{B}]
We will use the above proposition to show that $\mathcal{F}_{R,S}$ is weakly-modular. Indeed, the $(1,0)$-form $\phi_R(z)=e^{-|R(z)|}dz$ satisfies
$$\Big|\int_{\widehat{\mathbb{C}}}e
^{-|R(z)|}\omega_{FS}\Big|\leq\int_{\widehat{\mathbb{C}}}\omega_{FS}=\pi.$$ Therefore $\phi_R\in L^2_{(1,0)}$. Let $U\subset \widehat{\mathbb{C}}\setminus A_2(\Gamma_R)$. Then the deleted covering correspondence $\covr$ sends $z$ to the values $w$ for which $\frac{R(z)-R(w)}{z-w}=0$. Thus, any two local branches $f_1$ and $f_2$ of $\covr$ satisfy $f_1^*\phi_R(z)=R(z)=f_2^*\phi(z)$. By Proposition \ref{norm} part (b), this implies that $\Vert\frac{1}{(\deg R-1)}\covr^*\phi_R\Vert_{L^2}=\Vert\phi_R\Vert_{L^2}$ and $\Vert\frac{1}{(\deg R-1)}\covr^*\Vert=1$. Similarly, $\Vert\frac{1}{(\deg S-1)}\covs^*\Vert=1$. Since $d(\mathcal{F}_{R,S})=(\deg R-1)(\deg S-1)$ and $\mathcal{F}_{R,S}=\covr\circ\covs$, we get that
$$\Big\Vert\frac{1}{d(\mathcal{F}_{R,S})}\mathcal{F}^*_{R,S}\Big\Vert=1.$$
By Proposition \ref{norm} part (c), $\mathcal{F}_{R,S}$ is weakly-modular.

Now observe that
$$\left(\covr\circ\covs\right)(\Delta_S)\subseteq\covr\left(\interior\left(\Delta_R\right)\right)\subsetneq\interior\left(\Delta_S\right).$$
Suppose that $\mathcal{F}_{R,S}=\covr\circ\covs$ was modular. Then there would be a Borel measure $\lambda$ that assigns positive measure to nonempty open sets, and so that $(\mathcal{F}_{R,S})^*\lambda=(\deg(R)-1)(\deg(S)-1)\lambda$. In particular, we would have that 
$$\lambda(\interior(\Delta_S)\smallsetminus\covr\circ\covs(\interior(\Delta_S)))=0.$$
However,
$$\covr\circ\covs(\overline{\Delta_S})=\pi_2((\overline{\Delta_S}\times\widehat{\mathbb{C}})\cap\Gamma)$$
is closed, and hence $\interior(\Delta_S)\smallsetminus\covr\circ\covs(\overline{\Delta_S})$ is a nonempty open set which is contained in $\interior(\Delta_S)\smallsetminus\covr\circ\covs(\interior(\Delta_S))$. This is a contradiction. Thus, $\covr\circ\covs$ is not modular.
\end{proof}

Obtaining the measures $\mu_1$ and $\mu_2$ explicitly for the weak-modularity of $\mathcal{F}_{R,S}$ is possible following the argument in \cite{Mat23} for weak-modularity of $\F$. 


\subsection{Limit sets and rational behavior}

As mentioned in Section \ref{com}, $z=1$ is a fixed point of $\F$. Furthermore, it was proven in \cite{Mat23} that the set of critical values of $\F^{-1}$ is $B_1(\Gamma_a)=\lbrace\infty,-2,2\rbrace$. Thus, since $1\notin B_1(\Gamma_a)$, then there exists a holomorphic function $g_a:\Omega\to\widehat{\mathbb{C}}$ on a neighborhood $\Omega$ of $z=1$ whose graph contains $(1,1)$ and is contained in $\Gamma_a$. The following is from \cite{BulLom20I}.

\begin{proposition}
\begin{itemize}
\item[1.] After the change of coordinates $z\mapsto z-1$, $g_a$ has Taylor series expansion
$$g_a(z)=\left\lbrace\begin{matrix} z+\frac{a-7}{3(a-1)}z^2+\cdots&a\neq 7\\
z+\frac{1}{27}z^4+\cdots&a=7.
\end{matrix}\right.$$
\item[2.] Let $a\in\mathcal{K}$ and $(\Delta_{\J},\Delta_{\covq})$ be a Klein combination pair for $\F$. After a small perturbation of $\partial\Delta_{\J}$ and $\partial\Delta_{\covq}$ in a neighborhood of $z=1$, we obtain a Klein combination pair for $\F$ whose boundaries are smooth at $z=1$, and transverse to the repelling direction at $z=1$ of $g_a$ for $a\neq 7$, and to the real axis if $a=7$.
\end{itemize}
\end{proposition}

If $a\in\mathcal{K}$ and $(\Delta_{\J},\Delta_{\covq})$ is the modified Klein combination pair as in part 2 of the above proposition, then we define the \emph{forward} and \emph{backwards limit sets} of $\F$ as
$$\Lambda_{a,+}\coloneqq\bigcap\limits_{n=0}^{\infty}\F^n(\widehat{\mathbb{C}}\setminus\Delta_{\J})\mbox{ and }\Lambda_{a,-}\coloneqq\bigcap\limits_{n=1}^{\infty}\F^{-n}(\overline{\Delta_{\J}}),$$
respectively. Observe that  $\F=\J\circ\F^{-1}\circ\J$, and that $\Lambda_{a,+}=\J(\Lambda_{a,-})$.

We denote by $\F|$ the two-sided restriction $\F|:\Lambda_{a,-}\to\Lambda_{a,-}$, defined as $\F|(z)=\F(z)\cap\Lambda_{a,-}$ for $z\in\Lambda_{a,-}$. This is a well defined, $2$-to-$1$ map. In \cite{BulLom20I}, the authors showed that $\F|$ extends to a single-valued holomorphic map in a neighborhood of every point in $\partial\Lambda_{a,-}\setminus\lbrace -2\rbrace$, where $\lbrace -2\rbrace=\F^{-1}(1)\setminus\lbrace 1\rbrace$. We denote by $\f$ this extension. Furthermore, they proved the following.

\begin{theorem}[Bullett-Lomonaco]
After a surgery supported on $\widehat{\mathbb{C}}\setminus\Lambda_{a,-}$, the map $\f$ is hybrid equivalent to a parabolic quadratic rational map $P_A(z)=z+\frac{1}{z}+A$ on $\Lambda_{a,-}$, where $A\in\mathbb{C}$. That is, there exists a quasi-conformal map $h$ on a neighborhood of $\Lambda_{a,-}$ satisfying that $h\circ\f=P_A\circ h$ and $\overline{\partial}h=0$ on $\Lambda_{a,-}$.
\end{theorem}

Such map sends $\Lambda_{a,-}$ to the filled Julia set $K(P_A)$ of $P_A$, and hence $\partial\Lambda_{a,-}$ onto $\mathcal{J}(P_A)$, the Julia set of $P_A$. From \cite{FreLopMan83} and \cite{Lju83}, there exists a measure $\tilde{\mu}_A$ with $\supp(\tilde{\mu}_A)=\mathcal{J}(P_A)$ so that $\frac{1}{2^n}(P_A)^n\delta_{z_0}$ is weakly convergent to $\tilde{\mu}_A$, for all $z$ but at most $2$. Furthermore, $\tilde{\mu}_A$ is the measure of maximal entropy for $P_A$, having entropy $\log(2)$. We take $\mu_-\coloneqq h^*\tilde{\mu}_A$ and
$$\mathcal{E}_a\coloneqq\left\lbrace\begin{matrix}\emptyset&a\neq 5\\
\lbrace -1,2\rbrace&a=5.\end{matrix}\right.$$
It follows that $\supp(\mu_-)=\partial\Lambda_{a,-}$. In \cite{Mat23}, we prove:

\begin{theorem}
For all $z\in\widehat{\mathbb{C}}\setminus\mathcal{E}_a$, $\frac{1}{2^n}(\F^n)^*\delta_z$ is weakly convergent to $\mu_-$, and $\frac{1}{2^n}(\F^{-n})_*\delta_z$ is weakly convergent to $\mu_+\coloneqq\J^*\mu_-$, which is a measure supported on $\partial\Lambda_{a,+}$.
\end{theorem}

Similarly, given $(R,S)\in\mathcal{K}'$ and a Klein combination pair $(\Delta_R,\Delta_S)$ for $\mathcal{F}_{R,S}$, we define the \emph{forward} and \emph{backwards limit sets} of $\mathcal{F}_{R,S}$ as the sets
$$\Lambda_+\coloneqq\bigcap\limits_{n=1}^{\infty}\mathcal{F}_{R,S}^n(\Delta_S)\mbox{ and }\Lambda_-\coloneqq\bigcap\limits_{n=0}^{\infty}\mathcal{F}_{R,S}^{-n}(\Delta_R),$$
respectively. The next theorem is from \cite{Bul00}.

\begin{theorem}\label{KleinRS}
Let $(R,S)\in\mathcal{K}'$ having a Klein combination pair $(\Delta_R,\Delta_S)$. Then there exist rational maps $f_+$ and $f_-$, each of degree $(\deg(R)-1)(\deg(S)-1)$, so that the two sided restrictions $\mathcal{F}_{R,S}^{-1}|:\Lambda_+\to\Lambda_+$ and $\mathcal{F}_{R,S}|:\Lambda_-\to\Lambda_-$ are conjugate to the maps $f_+$  on $K(f_+)$ and $f_-$ on $K(f_-)$ respectively, conformally on interiors.
\end{theorem}
Here $K(f_+)$ and $K(f_-)$ denote the filled Julia sets of $f_+$ and $f_-$, respectively. Let $\tilde{\mu}_{f_+}$ be the measure of maximal entropy for the rational map $f_+$, and $\tilde{\mu}_{f_-}$ be the measure of maximal entropy for $f_-$. Let $h_+$ denote the conjugacy between $\mathcal{F}_{R,S}|^{-1}$ and $f_+$, and $h_-$ be the conjugacy between $\mathcal{F}|$ and $f_-$. Since $\supp(\tilde{\mu}_{f_\pm})=\mathcal{J}(f_{\pm})$, then $\mu^{\pm}\coloneqq(h_\pm)^*\mu_{f_{\pm}}$ is a measure whose support is $\supp(\mu^{\pm})=\partial\Lambda_\pm$.

In Section \ref{examples} we will prove that the equidistribution measure $\mu_-$ found in \cite{Mat23} is of maximal entropy for $\F$, just as in the case of rational maps. We also prove that the measure $\mu^-$ is of maximal entropy for $\mathcal{F}_{R,S}$.
\label{ent}
\subsection{Entropy for the families $\lbrace\F\rbrace_{a\in\mathcal{K}}$ and $\lbrace\mathcal{F}_{R,S}\rbrace_{R,S}$}\label{examples}
In this section, we prove Theorem \ref{A} and Theorem \ref{C}.

\begin{lemma}\label{measuresupport}
If $a\in\mathcal{K}$ and $\mu\in\mathcal{M}_1(\F)$, then $\supp(\mu)\subset\Lambda_{a,-}\cup\Lambda_{a,+}$.
\end{lemma}

\begin{proof}

Put $A=\overline{\Delta}_{\J}\setminus\Lambda_{a,-}$ and note that $(\F^{-n}(A))_{n=0}^\infty$ is a decreasing sequences of nested sets, whose intersection is empty. Furthermore, if $\mu$ is $\F$ invariant, then $(\mu(\F^{-n}(A)))_{n=0}^\infty$ is constant equal to $\mu(A)$. Therefore, $\mu(A)=0$. Now let $B\subset\widehat{\mathbb{C}}\setminus(\Lambda_{a,-}\cup\Lambda_{a,+})$ be compact. Then there exists $n_0\in\mathbb{N}$ such that $\F^{-n_0}(B)\subset A$, and thus,
$$\mu(B)\leq\mu(\F^{-n_0}(B))\leq\mu(A)=0.$$
This proves that $\supp\mu\subset\Lambda_{a,-}\cup\Lambda_{a,+}$.
\end{proof}

\begin{proof}[Proof of Theorem \ref{A}]
Metric entropy is preserved under ergodic equivalences, then 
$$H_{\mu_-}\left(\f,\bigvee\limits_{n=1}^{N-1}\mathcal{P}^n\right)=H_{\tilde{\mu}_A}\left(P_A,\bigvee\limits_{n=1}^{N-1}h(\mathcal{P})^n\right).$$
From Lemma \ref{restriction} and the above,
$$h_{\mu_-}(\F)\geq h_{\mu_-}(\f)=h_{\tilde{\mu}_A}(P_A)=\log(2),$$
as $\tilde{\mu}_A$ is the measure of maximal entropy of the rational map $P_A$ from \cite{FreLopMan83},\cite{Lju83}.
Using the Half-Variational Principle, Gromov's Inequality and Lemma \ref{entropy inequality}, we get that
$$\log(2)\leq h_{\mu_-}(\F)\leq\sup\limits_{\mu\in\mathcal{M}_1(\F)}h_{\mu}(\F)\leq h_{top}(\F)\leq h(\Gamma_a)\leq\log2.$$

Therefore, the measure $\mu_-$ maximizes entropy for the correspondence $\F$. From the conjugacy $\F\circ\J=\J\circ\F^{-1}$ and the fact that $\mu_+=\J^*\mu_-$, we have that $\mu_+$ is of maximal entropy for $\F^{-1}$.

Let $\mu$ be an invariant Borel probability measure for $\F$. By Lemma \ref{measuresupport} we know that $\supp\mu\subset\Lambda_{a,-}\cup\Lambda_{a,+}$.

Let $K^{\pm}\coloneqq \supp\mu\cap\Lambda_{a,\pm}$, and let $\mu^{\pm}$ be defined as 
$$\mu^{\pm}(A)=\frac{1}{\mu(\Lambda_{a,\pm})}\mu(A\cap\Lambda_{a,\pm}),$$
for all $A\in\mathcal{B}$. Write $\alpha_{\pm}\coloneqq \mu(\Lambda_{a,\pm})$. Then, $\mu^{\pm}$ is a Borel probability measure with $\supp\mu^{\pm}=K^{\pm}$, and 
$$\mu=\alpha_-\mu^-+\alpha_+\mu^+$$
with $\alpha_-+\alpha_+=1$.
\\

For a finite measurable partition $\mathcal{P}$, denote $\mathcal{P}^{\pm}=\lbrace P\in\mathcal{P}|P\subset\Lambda_{a,\pm}\rbrace$. Suppose $\mathcal{P}$ is a finite measurable partition, refined and ordered so that $\mathcal{P}=\mathcal{P}^-\cup\lbrace\lbrace 1\rbrace\rbrace\cup\mathcal{P}^+$. Then
$$\widetilde{\mathcal{P}}_n=\widetilde{(\mathcal{P}^-)}_n\cup\lbrace\lbrace 1\rbrace\rbrace\cup\left(\widetilde{(\mathcal{P}^+)}_n\vee\lbrace\Lambda_{a,+}\setminus\lbrace 1\rbrace\rbrace\right)=(\widetilde{\mathcal{P}}_n)^-\cup(\widetilde{\mathcal{P}}_n)^+.$$

It is easily shown that $\mathcal{P}'_n=(\mathcal{P}^-)'_n\cup\lbrace\lbrace 1\rbrace\rbrace\cup\left((\mathcal{P}^+)'_n\vee\lbrace\Lambda_{a,+}\setminus\lbrace 1\rbrace\right)$.  Thus,
\begin{eqnarray*}
h_{\mu}(\F,\mathcal{P})&=&\limsup\limits_{n\to\infty}\frac{1}{n}H_{\mu}(\mathcal{P}'_n)\\
&=&\limsup\limits_{n\to\infty}\frac{1}{n}\left[\sum\limits_{P\in\mathcal{P}'_n}-\mu(P)\log\mu(P)\right]\\
&=&\limsup\limits_{n\to\infty}\frac{1}{n}\left[\sum\limits_{P\in(\mathcal{P}^-)'_n}-\mu(P)\log\mu(P)-\mu(\lbrace 1\rbrace)\log(\lbrace 1\rbrace)+\sum\limits_{P\in(\mathcal{P}^+)'_n}-\mu(P)\log\mu(P)\right]
\end{eqnarray*}
We have that
$$h_{\mu}(\F.\mathcal{P})=\limsup\limits_{n\to\infty}\frac{1}{n}\left[\alpha_- H_{\mu^-}\left(\f,\left(\mathcal{P}^-\cup\lbrace 1\rbrace\rbrace\right)'_n\right)+\alpha_+H_{\mu^+}\left((\J\circ\f\circ\J)^{-1},\left(\mathcal{P}^+\cup\lbrace\lbrace 1\rbrace\rbrace\right)'_n\right)\right].$$
Since $\limsup\limits_{n\to\infty}\frac{1}{n} H_{\mu^-}\left(\f,\left(\mathcal{P}^-\cup\lbrace 1\rbrace\rbrace\right)'_n\right)=\lim\limits_{n\to\infty}\frac{1}{n}\alpha_- H_{\mu^-}\left(\f,\left(\mathcal{P}^-\cup\lbrace 1\rbrace\rbrace\right)'_n\right),$ then
$$h_{\mu}(\F.\mathcal{P})=\alpha_-h_{\mu^-}(\f,\left(\mathcal{P}^-\cup\lbrace\lbrace 1\rbrace\rbrace)\right)'_n+\alpha_+h_{\mu^+}
((\J\circ\f\circ\J)^{-1},\left(\mathcal{P}^+\cup\lbrace\lbrace 1\rbrace\rbrace)\right)'_n.$$

Taking supremum over all partitions $\mathcal{P}$, we obtain that 
$$h_{\mu}(\F)\geq\alpha_-h_{\mu^-}(\f)+\alpha_+h_{\mu^+}((\J\circ\f\circ\J)^{-1}).$$
Since $\f$ is quasiconformally conjugate to a quadratic rational map, then $h_{\mu^-}(\f)\leq\log 2$, and the equality holds only when $\mu^-=\mu_-$, since $\mu_-$ is precisely the pullback of the unique measure of maximal entropy of such rational map. From Lemma \ref{rational}, and the same conjugacy, we have that $h_{\mu^+}
((\J\circ\f\circ\J)^{-1},\left(\mathcal{P}^+\cup\lbrace\lbrace 1\rbrace\rbrace)\right)'_n=0.$

It follows that $h_{\mu}(\F)=\log2$ if and only if $\alpha_-=1$ and $\mu^-=\mu_-$. That is, $\mu=\mu_-$.
\end{proof}

\begin{lemma}\label{measuresupportcomposition}
If $(R,S)\in\mathcal{K}'$ and $\nu\in\mathcal{M}_1(\mathcal{F}_{R,S})$, then $\supp(\nu)\subset\Lambda_-\cup\Lambda_+$.
\end{lemma}
The proof of this lemma is the same as Lemma \ref{measuresupport} by taking $A=\overline{\Delta_{R}}\setminus\Lambda_-$ and $B\subset\widehat{\mathbb{C}}\setminus (\Lambda_-\cup\Lambda_+)$ compact.

\begin{proof}[Proof of Theorem \ref{C}]
We follow the proof of Theorem \ref{A}. The conjugacy between $\mathcal{F}|:\Lambda_-\to\Lambda_-$ and $f_-:\mathcal{J}(f_-)\to\mathcal{J}(f_-)$ gives us the inequality 
$$h_{\mu^-}(\mathcal{F}_{R,S})\geq h_{\mu_-}(f_-).$$
By \cite{FreLopMan83} and \cite{Lju83}, we have that $h_{\mu_{R,S}}(f_-)=\deg(f_-)=\log((\deg(R)-1)(\deg(S)-1))$. Making use of the Half-Variational Principle, Gromov's Inequality, and Lemma \ref{entropy inequality}, we obtain
$$h_{\mu^-}(\mathcal{F}_{R,S})=h_{top}(\mathcal{F}_{R,S})=\log((\deg(R)-1)(\deg(S)-1)).$$
Given that for $\nu\in\mathcal{M}_1(\mathcal{F}_{R,S})$ we have that $\supp\nu\subset\Lambda_-\cup\Lambda_+$ from Lemma \ref{measuresupportcomposition}, we have that for $\mathcal{P}=\mathcal{P}^-\cup\mathcal{P}^+$,
$$h_{\nu}(\mathcal{F}_{R,S},\mathcal{P})\geq\alpha h_{\nu^-}(\restr{\mathcal{F}_{R,S}}{\Lambda_-},\mathcal{P}^-)+(1-\alpha) h_{\nu^+}(\restr{\mathcal{F}_{R,S}}{\Lambda_+}^{-1},\mathcal{P}^+),$$
where $\nu=\alpha\nu^-+(1-\alpha)\nu^+$, and $\supp\nu^{\pm}\subset\Lambda_{\pm}$. By Theorem \ref{KleinRS} we have that $\restr{\mathcal{F}_{R,S}^{-1}}{\Lambda_+}$ is quasiconformally conjugate to the inverse of a rational map. Therefore, its metric entropy vanishes. Thus, $\alpha=1$ and the entropy maximizes only for $\nu=\mu^-$. Therefore, $\mu^-$ is the unique measure of maximal entropy for $\mathcal{F}_{R,S}$.
\end{proof}

\bibliographystyle{plain}
\bibliography{bibliography}

@article{BhaSri16,
	author = {Gautam Bharali and Shrihari Sridharan},
	doi = {10.1080/17476933.2016.1185419},
	journal = {Complex Variables and Elliptic Equations},
	language = {en},
	number = {12},
	pages = {1587-1613},
	publisher = {Taylor & Francis},
	title = {The dynamics of holomorphic correspondences of {$P^1$}: invariant measures and the normality set},
	url = {https://doi.org/10.1080/17476933.2016.1185419},
	volume = {61},
	year = {2016}}

@article{Bul00,
	author = {Bullett, Shaun},
	journal = {Conformal Geometry and Dynamics of the American Mathematical Society},
	number = {4},
	pages = {75--96},
	title = {A combination theorem for covering correspondences and an application to mating polynomial 		maps with Kleinian groups},
	volume = {4},
	year = {2000}}

@inproceedings{BulHar00,
	author = {S. R. Bullett and W. J. Harvey},
	booktitle = {Electronic Research Announcements of the AMS},
	pages = {21--30},
	title = {{Mating quadratic maps with Kleinian groups via quasiconformal surgery}},
	year = {2000}}

@article{BulLom20I,
	author = {Bullett, Shaun and Lomonaco, Luna},
	journal = {Inventiones mathematicae},
	number = {1},
	pages = {185--210},
	publisher = {Springer},
	title = {Mating quadratic maps with the modular group II},
	volume = {220},
	year = {2020}}

@article{BulPen94,
	author = {Bullett, Shaun and Penrose, Christopher},
	journal = {Inventiones mathematicae},
	number = {3},
	pages = {483-512},
	title = {Mating quadratic maps with the modular group},
	url = {http://eudml.org/doc/144179},
	volume = {115},
	year = {1994}}

@article {CarMetMor15,
    AUTHOR = {Carrasco-Olivera, Dante and Metzger Alvan, Roger and Morales
              Rojas, Carlos Arnoldo},
     TITLE = {Topological entropy for set-valued maps},
   JOURNAL = {Discrete Contin. Dyn. Syst. Ser. B},
  FJOURNAL = {Discrete and Continuous Dynamical Systems. Series B. A Journal
              Bridging Mathematics and Sciences},
    VOLUME = {20},
      YEAR = {2015},
    NUMBER = {10},
     PAGES = {3461--3474},
      ISSN = {1531-3492,1553-524X},
   MRCLASS = {37B40 (54C60)},
  MRNUMBER = {3411534},
MRREVIEWER = {James\ Pierre\ Kelly},
       DOI = {10.3934/dcdsb.2015.20.3461},
       URL = {https://doi.org/10.3934/dcdsb.2015.20.3461},
       }

@article{CloOhUll01,
	author = {Clozel, Laurent and Oh, Hee and Ullmo, Emmanuel},
	journal = {Inventiones mathematicae},
	number = {2},
	pages = {327--351},
	publisher = {Springer},
	title = {Hecke operators and equidistribution of Hecke points},
	volume = {144},
	year = {2001}}

@article{DinKauWu20,
	author = {Dinh, Tien-Cuong and Kaufmann, Lucas and Wu, Hao},
	journal = {International Journal of Mathematics},
	number = {05},
	pages = {2050036},
	publisher = {World Scientific},
	title = {{Dynamics of holomorphic correspondences on Riemann surfaces}},
	volume = {31},
	year = {2020}}

@article{DinSib08II,
	author = {Dinh, Tien-Cuong and Sibony, Nessim},
	journal = {Israel Journal of Mathematics},
	number = {1},
	pages = {29--44},
	publisher = {Springer},
	title = {Upper bound for the topological entropy of a meromorphic correspondence},
	volume = {163},
	year = {2008}}

@article{FreLopMan83,
	author = {Freire, Alexandre and Lopes, Artur and Man{\'e}, Ricardo},
	journal = {Boletim da Sociedade Brasileira de Matem{\'a}tica-Bulletin/Brazilian Mathematical Society},
	number = {1},
	pages = {45--62},
	publisher = {Springer},
	title = {An invariant measure for rational maps},
	volume = {14},
	year = {1983}}

@article{Gro03,
  title={On the entropy of holomorphic maps},
  author={Gromov, Mikha{\i}l},
  journal={Enseign. Math},
  volume={49},
  number={3-4},
  pages={217--235},
  year={2003}}

@article{KelTen17,
	author = {Kelly, James and Tennant, Tim},
	journal = {Houston J. Math.},
	number = {1},
	pages = {263--282},
	title = {Topological entropy on set-valued functions},
	volume = {43},
	year = {2017}}

@article{Lju83,
	author = {Ljubich, M Ju},
	journal = {Ergodic theory and dynamical systems},
	number = {3},
	pages = {351--385},
	publisher = {Cambridge University Press},
	title = {Entropy properties of rational endomorphisms of the Riemann sphere},
	volume = {3},
	year = {1983}}

@article{Mat23,
	author = {Matus de la Parra, V.},
	journal = {Ergodic Theory and Dynamical Systems},
	title = {Equidistribution for matings of quadratic maps with the modular group},
	year = {2023},
	pages={1-29},
	doi= {doi:10.1017/etds.2023.33}}

@article{VivSir22,
	author = {Vivas, Kendry J and Sirvent, V{\'\i}ctor F},
	journal = {Discrete and Continuous Dynamical Systems-B},
	publisher = {American Institute of Mathematical Sciences},
	title = {Metric entropy for set-valued maps},
	year = {2022}}
\end{document}